\documentclass{article}
\usepackage{graphicx}
\usepackage{amsmath}
\usepackage{amssymb}
\usepackage{amsthm}

\title{A Capillary Surface with No Radial Limits}
\author{C. Patric Mitchell \\
                       Department of Mathematics, Statistics \& Physics \\
                            Wichita State University \\
                            Wichita, Kansas, 67260-0033}

\def\Real{{\rm I\hspace{-0.2em}R}}
\def\Natural{{\rm I\hspace{-0.2em}N}}

\newcommand\myeq{\mathrel{\overset{\makebox[0pt]{\mbox{\normalfont\tiny\sffamily def}}}{=}}}

\date{ }

\begin{document}
\maketitle

\vspace*{3mm}

\begin{abstract}
In 1996,  Kirk Lancaster and David Siegel investigated the existence and behavior of radial limits at a corner 
of the boundary of the domain of solutions of capillary and other prescribed mean curvature problems
with contact angle boundary data.  In Theorem 3, they provide an example of a capillary surface in a unit disk $D$  which has no radial limits 
at $(0,0)\in\partial D.$   In their example, the contact angle ($\gamma$)  cannot be bounded away from zero and $\pi.$  
Here we consider a domain $\Omega$  with a convex corner at $(0,0)$  and find a capillary surface $z=f(x,y)$  in $\Omega\times\Real$  
which has no radial limits at $(0,0)\in\partial\Omega$  such that $\gamma$  is bounded away from $0$  and $\pi.$   
\end{abstract}

\newtheorem{thm}{Theorem}
\newtheorem{prop}{Proposition}
\newtheorem{cor}{Corollary}
\newtheorem{lem}{Lemma}
\newtheorem{rem}{Remark}
\newtheorem{example}{Example}

\section{Introduction}
Let $\Omega$  be a domain  in ${\Real}^{2}$  with locally Lipschitz boundary  and ${\cal O}=(0,0)\in \partial\Omega$  
such that $\partial\Omega\setminus \{ {\cal O} \}$  is a $C^{4}$  curve and $\Omega\subset B_{1}\left(0,1\right),$  
where $B_{\delta}\left({\cal N}\right)$  is the open ball in $\Real^{2}$  of radius $\delta$ about ${\cal N}\in {\Real}^{2}.$
Denote the unit exterior normal to $\Omega$  at $(x,y)\in\partial\Omega$  by $\nu(x,y)$  and let polar coordinates relative to ${\cal O}$  
be denoted by $r$ and $\theta.$   We shall assume there exists a $\delta^{*}\in (0,2)$  and $\alpha \in \left(0,\frac{\pi}{2}\right)$  
such that $\partial \Omega \cap B_{\delta^{*}}({\cal O})$   consists of the line segments 
\[
{\partial}^{+}\Omega^{*} = \{(r\cos(\alpha),r\sin(\alpha)):0\le r\le \delta^{*}\}
\]  
and 
\[
{\partial}^{-}\Omega^{*} = \{(r\cos(-\alpha),r\sin(-\alpha)):0\le r\le \delta^{*}\}.
\]
Set $\Omega^{*} = \Omega \cap B_{\delta^{*}}({\cal O}).$  
Let $\gamma:\partial\Omega\to [0,\pi]$  be given.  Let $\left(x^{\pm}(s),y^{\pm}(s)\right)$  be arclength parametrizations of 
$\partial^{\pm}\Omega$  with $\left(x^{+}(0),y^{+}(0)\right)=\left(x^{-}(0),y^{-}(0)\right)=(0,0)$  and set 
$\gamma^{\pm}(s)=\gamma\left(x^{\pm}(s),y^{\pm}(s)\right).$  

Consider the capillary problem of finding a function $f\in C^2(\Omega)\cap C^1(\overline{\Omega}\setminus\{{\cal O}\})$  
satisfying 
\begin{equation}
\label{CAP}
{\rm div}(Tf)=\frac{1}{2}f\ \ {\rm in}\ \Omega
\end{equation}
and 
\begin{equation}
\label{CAPBC}
Tf\cdot\nu=\cos\left(\gamma\right)\ {\rm on}\ \partial\Omega\setminus\{{\cal O}\},  
\end{equation}
where $Tf=\frac{\nabla f}{\sqrt{1+|\nabla f|^{2}}}.$   
We are interested in the existence of the radial limits $Rf(\cdot)$   of a solution $f$  of (\ref{CAP})--(\ref{CAPBC}),  where
\[
Rf(\theta)=\lim_{r \rightarrow 0^{+}} f(r\cos \theta, r\sin \theta), -\alpha < \theta < \alpha
\]
and $Rf(\pm \alpha)=\lim_{\partial^{\pm}\Omega^{*}\ni {\bf x} \rightarrow {\cal O}} f({\bf x}), {\bf x} = (x,y)$, 
which are the limits of the boundary values of $f$ on the two sides of the corner if these exist.  
In \cite{CEL1}, the following is proven: 

\begin{prop}
\label{ONE}
Let $f$ be a bounded solution to (\ref{CAP}) satisfying (\ref{CAPBC}) on $\partial^{\pm}\Omega^{*} \setminus \{{\cal O}\}$ 
which is discontinuous at ${\cal O}.$   If $\alpha > \pi/2$ then $Rf(\theta)$ exists for all $\theta \in (-\alpha,\alpha).$   
If $\alpha \le \pi/2$ and there exist constants $\underline{\gamma}^{\, \pm},
\overline{\gamma}^{\, \pm}, 0 \le \underline{\gamma}^{\, \pm} \leq \overline{\gamma}^{\, \pm} \le \pi,$   satisfying
\[ 
\pi - 2\alpha  < \underline{\gamma}^{+} + \underline{\gamma}^{-} \le  
\overline{\gamma}^{\, +} +  \overline{\gamma}^{\, -} < \; \pi + 2\alpha 
\]
so that $\underline{\gamma}^{\pm}\leq \gamma^{\pm}(s) \leq \overline{\gamma}^{\,  \pm}$
for all $s, 0<s<s_{0},$ for some $s_{0}$, then again $Rf(\theta)$ exists for
all $\theta \in (-\alpha, \alpha)$. 
\end{prop}

\noindent In \cite{LS1}, Lancaster and Siegel proved this theorem with the additional restriction that $\gamma$  be bounded away from 
$0$  and $\pi;$  Figure \ref{FigOne} illustrates these cases.  

\begin{figure}[ht]
\centering
\includegraphics{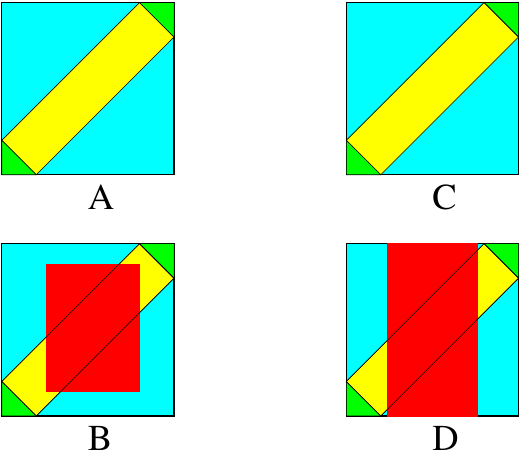}
\caption{The Concus-Finn rectangle (A \& C) with regions ${\cal R}$ (yellow), $D_{2}^{\pm}$ (blue) and $D_{1}^{\pm}$ (green);  the 
restrictions on $\gamma$  in \cite{LS1} (red region in B) and in \cite{CEL1} (red region in D). \label{FigOne} }
\end{figure} 

\noindent In Theorem 3 of \cite{LS1}, Lancaster and Siegel also proved 

\begin{prop}
\label{TWO}
Let $\Omega$ be the disk of radius $1$ centered at $(1,0).$
Then there exists a solution to $Nf = \frac{1}{2} f$ in $\Omega,
|f| \leq 2, f \in C^{2}(\Omega) \cap C^{1}( \overline{\Omega} \setminus O),
O = (0,0)$ so that no radial limits $Rf(\theta)$ exist ($\theta \in [ -\pi/2,\pi/2 ]$).
\end{prop}

\noindent In this case, $\alpha=\frac{\pi}{2};$  if $\gamma$  is bounded away from $0$  and $\pi,$  then Proposition \ref{ONE}
would imply that $Rf(\theta)$  exists for each  $\theta \in \left[-\frac{\pi}{2},\frac{\pi}{2}\right]$  and therefore 
the contact angle $\gamma = \cos^{-1}\left(Tf\cdot\nu\right)$  in Proposition \ref{TWO} is not bounded away from $0$  and $\pi.$

In our case, the domain $\Omega$  has a convex corner of size $2\alpha$  at ${\cal O}$  and we wish to investigate the question 
of whether an example like that in Proposition \ref{TWO} exists in this case when $\gamma$  is bounded away from $0$  and $\pi.$  
In terms of the Concus-Finn rectangle, the question is whether, given $\epsilon>0,$  there is a 
$f\in C^2(\Omega)\cap C^1(\overline{\Omega}\setminus\{ {\cal O} \} )$  of (\ref{CAP})--(\ref{CAPBC}) such that 
no radial limits $Rf(\theta)$ exist ($\theta \in [-\alpha,\alpha]$)  and $\vert\gamma-\frac{\pi}{2}\vert\le \alpha+\epsilon;$
this is illustrated in Figure \ref{FigTwo}.

\begin{figure}[ht]
\centering
\includegraphics{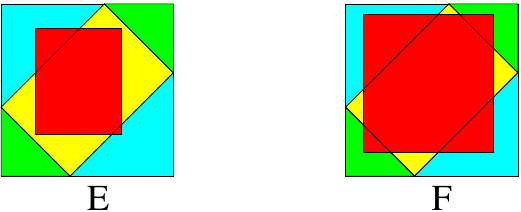}
\caption{The Concus-Finn rectangle. When $\gamma$  remains in red region in E, $Rf(\cdot)$  exists; $\gamma$  in Theorem \ref{THREE} 
remains in the red region in F . \label{FigTwo} }
\end{figure} 

\begin{thm}  
\label{THREE}
For each $\epsilon>0$, there is a domain $\Omega$ as described above and a solution 
$f\in C^2(\Omega)\cap C^1(\overline{\Omega}\setminus\{ {\cal O} \} )$  of (\ref{CAP}) such that 
the contact angle $\gamma = \cos^{-1}\left(Tf \cdot \nu\right):\partial\Omega\setminus\{ {\cal O}\} \to [0,\pi]$  
satisfies $\vert\gamma-\frac{\pi}{2}\vert\le \alpha+\epsilon$  and there exist a sequence $\{r_{j}\}$  in $(0,1)$  
with $\lim_{j\to\infty} r_{j}=0$  such that 
\[
(-1)^{j}f\left(r_{j},0\right)>1 \ \ \ \ {\rm for \ each \ } j\in\Natural.
\]
Assuming $\Omega$  and $\gamma$  are symmetric with respect to the line $\{(x,0):x\in\Real\},$  this implies that  
no radial limit 
\begin{equation}
\label{Rads}
Rf(\theta) \myeq \lim_{r\downarrow 0} f(r\cos(\theta),r\sin(\theta))
\end{equation} 
exists for any $\theta\in[-\alpha,\alpha].$   
\end{thm}

\noindent We note that our Theorem is an extension of Theorem 3 of \cite{LS1} to contact angle data in a domain with a convex corner. 
As in \cite{Lan:89,LS1}, we first state and prove a localization lemma; this is analogous to the Lemma in \cite{Lan:89} and 
Lemma 2 of \cite{LS1}.

\begin{lem}
\label{LEM}
Let $\Omega\subseteq{\Real}^{2}$  be as above, $\epsilon >0,$  $\eta>0$
and $\gamma_{0}:\partial\Omega\setminus\{ {\cal O}\} \to [0,\pi]$  
such that  $\vert\gamma_{0}-\frac{\pi}{2}\vert\le \alpha+\epsilon.$
For each $\delta\in(0,1)$   and $h\in C^2(\Omega)\cap C^1(\overline{\Omega}\setminus\{{\cal O}\})$  which satisfies (\ref{CAP}) and 
(\ref{CAPBC})  with $\gamma=\gamma_{0},$   
there exists a solution $g\in C^2(\Omega)\cap C^1(\overline{\Omega}\setminus\{{\cal O}\})$  of (\ref{CAP}) such that 
$\lim_{\overline\Omega\ni (x,y)\to (0,0) } g(x,y)=+\infty,$   
\begin{equation}
\label{SUP}
\sup_{\Omega_{\delta}}  \vert g-h\vert<\eta \ \ \ \ {\rm and} \ \ \ \ \left\vert\gamma_{g}-\frac{\pi}{2}\right\vert\le \alpha+\epsilon,
\end{equation}
where $\Omega_{\delta} = \overline\Omega\setminus B_{\delta}\left({\cal O}\right)$  and 
$\gamma_{g}=\cos^{-1}\left(Tg\cdot \nu\right):\partial\Omega\setminus\{{\cal O}\}\to [0,\pi]$  is the contact angle 
which the graph of $g$  makes with $\partial\Omega\times\Real.$  
\end{lem}

\begin{proof}
Let $\epsilon, \eta, \delta,\Omega,h$  and $\gamma_0$ be given. For $\beta\in(0,\delta)$, let 
$g_{\beta}\in C^{2}\left(\Omega)\cap C^{1}(\overline{\Omega}\setminus\{ {\cal O}\}\right)$  satisfy (\ref{CAP}) and (\ref{CAPBC}) 
with $\gamma=\gamma_{\beta},$  where 
\[
\gamma_{\beta}= \left\{ \begin{array}{ccc} 
\frac{\pi}{2}-\alpha-\epsilon & {\rm on} & \overline{B_{\beta}\left({\cal O}\right)}\\
\gamma_{0} & {\rm on} &  \overline{\Omega}\setminus B_{\beta}\left({\cal O}\right).\\ 
\end{array}
\right.
\]
As in the proof of Theorem 3 of \cite{LS1}, $g_{\beta}$  converges to $h,$  pointwise and uniformly in the $C^{1}$  norm on 
$\overline{\Omega_{\delta}}$  as $\beta$  tends to zero.  
Fix $\beta>0$  small enough that  $\sup_{\Omega_{\delta}}  \vert g-h\vert<\eta.$  

Set $\Sigma= \{(r\cos(\theta),r\sin(\theta)): r>0, -\alpha\le \theta \le \alpha \}.$  
Now define $w:\Sigma\to \Real$  by 
\[
w(r\cos \theta,r\sin\theta) = \frac{\cos\theta-\sqrt{k^{2}-\sin^{2}\theta}}{k\kappa r},
\]
where $k=\sin\alpha \sec\left(\frac{\pi}{2}-\alpha-\epsilon\right) = \sin\alpha \csc(\alpha+\epsilon).$
As in \cite{CF}, there exists a $\delta_{1}>0$  such that ${\rm div}(Tw)-\frac{1}{2}w\ge 0$  on $\Sigma\cap B_{\delta_{1}}({\cal O}),$  
$Tw\cdot\nu=\cos\left(\frac{\pi}{2}-\alpha-\epsilon\right)$  on $\partial\Sigma \cap B_{\delta_{1}}({\cal O}),$  and   
$\lim_{r\to 0^{+}} w(r\cos \theta,r\sin\theta) = \infty$  for each $\theta\in[-\alpha,\alpha].$  
We may assume $\delta_{1} \le \delta^{*}.$   Let 
\[
M=\sup_{\Omega\cap \partial B_{\delta_{1}}({\cal O})} |w-g_{\beta}| \ \ \ {\rm and} \ \ \ w_{\beta}=w-M.
\]
Since ${\rm div}(Tw_{\beta})-\frac{1}{2}w_{\beta}\ge \frac{M}{2}\ge 0={\rm div}(Tg_{\beta})-\frac{1}{2}g_{\beta}$  
in $\Omega\cap B_{\delta_{1}}({\cal O}),$   $w_{\beta}\le g_{\beta}$  on $\Omega\cap\partial B_{\delta_{1}}({\cal O})$  
and $Tg_{\beta}\cdot\nu\ge Tw_{\beta}\cdot\nu$  on $\partial\Omega\cap B_{\delta_{1}}({\cal O}),$  
we see that $g_{\beta}\ge w_{\beta}$  on $\Omega\cap\partial B_{\delta_{1}}({\cal O}).$  
\end{proof}

\noindent  We may now prove Theorem \ref{THREE}.

\begin{proof}
We shall construct a sequence ${f_n}$  of solutions of (\ref{CAP}) and a sequence $\{r_{n}\}$  of positive real numbers 
such that $\lim_{n\to\infty} r_{n}=0,$  $f_{n}(x,y)$  is even in $y$  and  
\[
(-1)^{j}f_{n}\left(r_{j},0\right)>1 \ \ \ \ {\rm for \ each \ } j=1,\dots,n.
\]

Let $\gamma_{0}=\frac{\pi}{2}$  and $f_{0}=0.$   Set $\eta_{1}=1$  and $\delta_{1}=\delta_{0}.$  
From Lemma \ref{LEM}, there exists a $f_{1}\in C^2(\Omega)\cap C^1(\overline{\Omega}\setminus\{ {\cal O}\})$  which satisfies (\ref{CAP})
such that $\sup_{\Omega_{\delta_{1}}} \vert f_{1}-f_{0}\vert<\eta_{1},$  $\left\vert\gamma_{1}-\frac{\pi}{2}\right\vert\le \alpha+\epsilon$  
and $\lim_{\Omega\ni (x,y)\to {\cal O}} f_{1}(x,y)=-\infty,$  where $\gamma_{1}=\cos^{-1}\left(Tf_{1}\cdot \nu \right).$  
Then there exists $r_{1}\in \left(0,\delta_{1}\right)$  such that $f_{1}\left(r_{1},0\right)<-1.$  

Now set $\eta_{2}=-\left(f_{1}\left(r_{1},0\right)+1\right)>0$  and $\delta_{2}=r_{1}.$  
From Lemma \ref{LEM}, there exists a $f_{2}\in C^2(\Omega)\cap C^1(\overline{\Omega}\setminus\{ {\cal O}\})$  which satisfies (\ref{CAP})
such that $\sup_{\Omega_{\delta_{2}}} \vert f_{2}-f_{1}\vert<\eta_{2},$  $\left\vert\gamma_{2}-\frac{\pi}{2}\right\vert\le \alpha+\epsilon$  
and $\lim_{\Omega\ni (x,y)\to {\cal O}} f_{2}(x,y)=\infty,$  where $\gamma_{2}=\cos^{-1}\left(Tf_{2}\cdot \nu \right).$  
Then there exists $r_{2}\in \left(0,\delta_{2}\right)$  such that $f_{2}(r_{2},0)>1.$  
Since $(r_{1},0)\in \Omega_{\delta_{2}},$  
\[
f_{1}\left(r_{1},0\right)+1<f_{2}\left(r_{1},0\right)-f_{1}\left(r_{1},0\right)<-\left(f_{1}\left(r_{1},0\right)+1 \right)
\]
and so $f_{2}\left(r_{1},0\right)<-1.$

Next set $\eta_{3}=\min\left\{-\left(f_{2}\left(r_{1},0\right)+1\right), f_{2}\left(r_{2},0\right)-1\right\}>0$  
and $\delta_{3}=r_{2}.$  
From Lemma \ref{LEM}, there exists a $f_{3}\in C^2(\Omega)\cap C^1(\overline{\Omega}\setminus\{ {\cal O}\})$  which satisfies (\ref{CAP})
such that $\sup_{\Omega_{\delta_{3}}} \vert f_{3}-f_{2}\vert<\eta_{3},$  $\left\vert\gamma_{3}-\frac{\pi}{2}\right\vert\le \alpha+\epsilon$  
and $\lim_{\Omega\ni (x,y)\to {\cal O}} f_{3}(x,y)=-\infty,$  where $\gamma_{3}=\cos^{-1}\left(Tf_{3}\cdot \nu \right).$  
Then there exists $r_{3}\in \left(0,\delta_{3}\right)$  such that $f_{3}(r_{3},0)<-1.$  
Since $(r_{1},0),(r_{2},0)\in \Omega_{\delta_{2}},$  we have  
\[
f_{2}\left(r_{1},0\right)+1<f_{3}\left(r_{1},0\right)-f_{2}\left(r_{1},0\right)<-\left(f_{2}\left(r_{1},0\right)+1\right)
\]  
and  
\[
-\left(f_{2}\left(r_{2},0\right)-1\right)<f_{3}\left(r_{2},0\right)-f_{2}\left(r_{2},0\right)<f_{2}\left(r_{2},0\right)-1; 
\]
hence $f_{3}\left(r_{1},0\right)<-1$  and $1<f_{3}\left(r_{2},0\right).$

Continuing to define $f_{n}$ and $r_{n}$ inductively,  we set 
\[
\eta_{n+1} = \min_{1\leq j\leq n}\vert f_n(r_j,0)-(-1)^j \vert \ \ \ {\rm and} \ \ \  
\delta_{n+1}=\min\left\{r_{n},\frac{1}{n}\right\}.
\]
From Lemma  \ref{LEM}, there exists $f_{n+1}\in C^2(\Omega)\cap C^1(\overline{\Omega}\setminus\{ {\cal O}\})$  which satisfies (\ref{CAP})  
such that $\sup_{\Omega_{\delta_{n+1}}} \vert f_{n+1}-f_{n}\vert<\eta_{n+1},$  
$\left\vert\gamma_{n+1}-\frac{\pi}{2}\right\vert\le \alpha+\epsilon$  
and $\lim_{\Omega\ni (x,y)\to {\cal O}} f_{n+1}(x,y)=(-1)^{n+1}\infty,$  where $\gamma_{n+1}=\cos^{-1}\left(Tf_{n+1}\cdot \nu \right).$  
Then there exists $r_{n+1}\in  \left(0,\delta_{n+1}\right)$  such that $(-1)^{n+1}f_{n+1}(r_{n+1},0)>1.$  
For each $j\in \{1,\dots,n\}$  which is an even number, we have 
\[
-\left(f_{n}\left(r_{j},0\right)-1\right)<f_{n+1}\left(r_{j},0\right)-f_{n}\left(r_{j},0\right)<f_{n}\left(r_{j},0\right)-1
\]
and so $1<f_{n+1}\left(r_{j},0\right).$  For each $j\in \{1,\dots,n\}$  which is an odd number, we have
\[
f_{n}\left(r_{j},0\right)+1<f_{n+1}\left(r_{j},0\right)-f_{n}\left(r_{j},0\right)<-\left(f_{n}\left(r_{j},0\right)+1\right)
\]
and so $f_{n+1}\left(r_{j},0\right)<-1.$

As in \cite{LS1,Siegel}, there is a subsequence of $\{f_{n}\},$   still denoted  $\{f_{n}\},$  which converges pointwise and 
uniformly in the $C^{1}$  norm on $\overline{\Omega_{\delta}}$  for each $\delta>0$  as $n\to\infty$  to a solution 
$f\in C^2(\Omega)\cap C^1\left(\overline\Omega\setminus {\cal O}\right)$  of  (\ref{CAP}).
For each $j\in\Natural$  which is even, $f_{n}\left(r_{j},0\right)>1$  for each $n\in\Natural$  and so $f\left(r_{j},0\right)\ge 1.$  
For each $j\in\Natural$  which is odd, $f_{n}\left(r_{j},0\right)<-1$  for each $n\in\Natural$  and so $f\left(r_{j},0\right)\le -1.$
Therefore
\[
\lim_{r\to 0^{+}} f(r,0)  \ \ {\rm does \ not\ exist, \ even \ as\ an \ infinite\ limit},
\]
and so $Rf(0)$  does not exist.

Since $\Omega$  is symmetric with respect to the $x-$axis and $\gamma_{n}(x,y)$  is an even function of $y,$  
$f(x,y)$  is an even function of $y.$ 
Now suppose that there exists $\theta_{0}\in [-\alpha,\alpha]$  such that $Rf(\theta_0)$  exists; then  $\theta_{0}\neq 0.$  
From the symmetry of $f,$   $Rf(-\theta_{0})$  must also exist and $Rf(-\theta_{0})=Rf(\theta_{0})$. 
Set $\Omega'=\{(r\cos\theta,r\sin\theta): 0<r<\delta_{0},-\theta_{0}<\theta<\theta_{0}\}\subset \Omega.$ 
Since $f$  has continuous boundary values on $\partial\Omega',$   $f\in C^{0}\left(\overline{\Omega'}\right)$  
and so $Rf(0)$  does exist, which is a contradiction.  Thus $Rf(\theta)$  does not exist for any $\theta\in [-\alpha,\alpha].$
\end{proof}

\end{document}